\documentclass[12pt]{article}
\usepackage{times}
\usepackage{amsmath,amsthm,amsfonts,amssymb}
\usepackage{txfonts}
\usepackage{cite}
\usepackage{indentfirst}
\usepackage{mathtools}
\usepackage[colorlinks=true,anchorcolor=blue,filecolor=blue,
linkcolor=blue,urlcolor=blue,citecolor=blue]{hyperref}
\evensidemargin=\oddsidemargin\topmargin=-1.5cm

\newtheorem{thm}{Theorem}[section]
\newtheorem{lem}[thm]{Lemma}
\newtheorem{cor}[thm]{Corollary}

\newtheorem{conj}{Conjecture}[section]
\theoremstyle{definition}
\newtheorem{definition}[thm]{Definition}
\theoremstyle{remark}

\def \A{\mathcal{A}}

\begin{document}

\title{Tensor Spectral Stability for Uniform Hypergraphs with Bounded Matching Number}
\author{Yi Xu, Yi-Zheng Fan\thanks{Corresponding author. Supported by National Natural Science Foundation of China (No. 12331012).
Email: yxu123@stu.ahu.edu.cn, fanyz@ahu.edu.cn}\\
\small \it Center for Pure Mathematics, School of Mathematical Sciences, \\ \small \it Anhui University, Hefei 230601, P. R. China}
\date{}
\maketitle

\begin{abstract}
We establish a tensor spectral stability theorem for uniform hypergraphs with bounded matching number.
More precisely, for fixed integers $k\geq 3$ and $\beta\geq2$, and sufficiently large $n$, we prove that every $n$-vertex $k$-uniform hypergraph $H$ with matching number at most $\beta$ and tensor spectral radius close to the maximum possible value among all such hypergraphs must be structurally close to the extremal hypergraph $S_{n,k,\beta}$, whose edges consist of all $k$-sets intersecting a fixed set of $\beta$ vertices.
Furthermore, we show that every edge of $H$ intersects this distinguished vertex set and that $H$ contains all but a small proportion of the edges of $S_{n,k,\beta}$.
As an application, we obtain a new proof of the spectral version of the Erdős matching conjecture for sufficiently large $n$.

\noindent
\textbf{Keywords}: Uniform hypergraph; adjacency tensor; spectral radius; stability; matching number

\noindent
\textbf{MSC 2020}: 05C65, 05C35, 15A69
\end{abstract}

\section{Introduction}
A hypergraph $H=(V(H),E(H))$ consists of a vertex set $V(H)$ and an edge set $E(H)$, where each edge is a subset of $V(H)$.
If every edge has size $k$, then $H$ is called a \emph{$k$-uniform hypergraph}, or simply a \emph{$k$-graph}.
A \emph{matching} in $H$ is a collection of pairwise disjoint edges.
The \emph{matching number} of $H$, denoted by $\mu(H)$, is the maximum size of matchings in $H$.

A central problem in extremal hypergraph theory is to determine the maximum number of edges in a uniform hypergraph with bounded matching number.
In 1965, Erd{\H{o}}s~\cite{Erdos1965} proposed the following conjecture.

\begin{conj}\label{conj:EMC}
Let $k\geq 2$, $\beta\geq 1$, and $n\geq k(\beta+1)-1$ be integers.
If $H$ is an $n$-vertex $k$-graph with $\mu(H)\leq\beta$, then
$$
|E(H)|\leq\max\left\{\binom{k(\beta+1)-1}{k},\binom{n}{k}-\binom{n-\beta}{k}\right\}.
$$
\end{conj}

The conjecture is known to hold for $k=2$~\cite{ErdosGallai1959} and $k=3$~\cite{LuczakMieczkowska2014,Frankl2017}.
For a general $k$, Erd{\H{o}}s~\cite{Erdos1965} proved the conjecture when $n$ is sufficiently large.
The lower bound on $n$ was later improved in a series of works; see, for example, \cite{Frankl2013,Frankl2017,FranklKupavskii2022}.
Nevertheless, the conjecture remains open in full generality.

Spectral extremal problems ask for the maximum spectral radius over a class of graphs or hypergraphs satisfying a given combinatorial
restriction.
In 2005, Qi~\cite{Qi2005} and Lim~\cite{Lim2005} independently introduced the eigenvalues of tensors.
In 2012, Cooper and Dutle~\cite{CooperDutle2012} defined the adjacency tensor of a uniform hypergraph and developed a spectral theory for uniform hypergraphs.
These definitions make it possible to study spectral extremal problems for uniform hypergraphs.
In this paper, the spectral radius of a uniform hypergraph always refers to the spectral radius of the adjacency tensor of the hypergraph.

Keevash, Lenz, and Mubayi~\cite{KeevashLenzMubayi2014} gave two general criteria for deriving spectral extremal results from stability results for the corresponding classical extremal problems.
Cooper, Desai, and Sahay~\cite{CooperDesaiSahay} studied the principal eigenvectors in hypergraph Tur\'an problems and showed that their entries are close
to being equal for a general class of extremal hypergraphs.
She, Fan, and Kang~\cite{SheFanKang} established a relation between the spectral radius of a linear hypergraph and its local structure through walks.
More recently, Zheng, Li, and Fan~\cite{ZhengLiFan} extended the general results of Keevash, Lenz, and Mubayi and obtained spectral Tur\'an results for several classes of uniform hypergraphs.
For the matching problem in graphs, Feng, Yu, and Zhang~\cite{FengYuZhang2007} determined the maximum spectral radius of an $n$-vertex graph with a given matching number and characterized the extremal graphs.

Nikiforov~\cite{NikiforovStability} proved a spectral stability theorem for forbidden-subgraph problems.
For problems involving bounded matching number, Jiang, Yuan, and Zhai~\cite{JiangYuanZhai} used a stability result to characterize exact spectral extremal graphs.
Wu, Kang, and Ni~\cite{WuKangNi} later proved a general spectral stability theorem that applies, in particular, to graphs with bounded matching number.
For uniform hypergraphs, Liu, Ni, Wang, and Kang~\cite{LiuNiWangKang} established $p$-spectral versions of stability theorems for two classes of expanded hypergraphs.
These results motivate the study of spectral stability for uniform hypergraphs with bounded matching number.

Recently, by using shifting method, Kang, Lu, Yuan, and Zhou \cite{KangLuYuanZhou2026} determined the maximum spectral radius of an $n$-vertex $k$-graph with matching number at most $\beta$ for sufficiently large $n$ and characterized the unique extremal hypergraph.
%Their proof is based on the shifting method.
In this paper, we develop a spectral stability method based on normalized nonnegative eigenvectors.
Unlike the shifting approach, our method directly exploits the distribution of the principal eigenvector and may be applicable to other spectral extremal problems for uniform hypergraphs.
As a direct application, our result implies the spectral version of the Erd{\H{o}}s matching conjecture for sufficiently large $n$.

Let $[n]:=\{1,2,\ldots,n\}$.
Let $S_{n,k,\beta}$ be the $k$-graph on $[n]$ whose edge set consists of all $k$-sets intersecting $[\beta]$; that is,
$$
E(S_{n,k,\beta})=\bigg\{e\in\binom{[n]}{k}:e\cap[\beta]\neq\emptyset\bigg\}.
$$
The spectral radius of $S_{n,k,\beta}$ is denoted by $\rho(S_{n,k,\beta})$.
For two sets $A$ and $B$, write $A \Delta B:=(A \setminus B) \cup (B \setminus A)$.
Our main result is stated as follows.

\begin{thm}
\label{thm:spectral-stability}
Let $k\geq3$ and $\beta\geq2$ be fixed integers.
There exist constants $C=C(k,\beta)>0$ and $n_0$ such that the following holds whenever $n\geq n_0$, and $\delta\geq0$ is sufficiently small.
If $H$ is a $k$-graph on $[n]$ satisfying $\mu(H)\leq\beta$ and
$$
\rho(H)\geq(1-\delta)\rho(S_{n,k,\beta}),
$$
then  there exists $W\in\binom{[n]}{\beta}$ such that every edge of $H$ intersects $W$.
After relabeling the vertices of $H$, we have
$$
\left|E(H)\triangle E(S_{n,k,\beta})\right|\leq C\delta n^{k-1}.
$$
\end{thm}

Taking $\delta=0$ in Theorem~\ref{thm:spectral-stability} gives the following tensor spectral extremal result of Kang, Lu, Yuan, and
Zhou~\cite{KangLuYuanZhou2026}.

\begin{cor}\label{cor:spec-emc}
Let $k\geq3$ and $\beta\geq2$ be fixed.
For sufficiently large $n$, if $H$ is an $n$-vertex $k$-graph with $\mu(H)\leq\beta$, then
$$
\rho(H)\leq\rho(S_{n,k,\beta}).
$$
Moreover, equality holds if and only if $H\cong S_{n,k,\beta}$.
\end{cor}

The remainder of the paper is organized as follows.
In Section~2, we introduce some notions and estimates used in the proof.
In Section~3, we first find the set $W$ in Theorem~\ref{thm:spectral-stability},
  and then estimate the number of edges missing from $S_{n,k,\beta}$.
These two steps complete the proof of main theorem.

\section{Preliminaries}\label{sec2}
An order-$k$ and dimension-$n$ \emph{tensor} $\A$ over $\mathbb{C}$ is a multidimensional array $\A=(a_{i_1i_2\ldots i_k})$ with $n^k$ complex entries, where $i_j\in[n]$ for each $j\in[k]$.
A tensor is called \emph{symmetric} if its entries are invariant under any permutation of the indices.
The notions of tensor eigenvalues and eigenvectors were introduced independently by Qi \cite{Qi2005} and Lim \cite{Lim2005}.

\begin{definition}
Let $\A=(a_{i_1i_2\cdots i_k})$ be an order-$k$ and dimension-$n$ tensor, where $k\geq2$.
If there exist a complex number $\lambda \in \mathbb{C}$ and a nonzero vector $x=(x_1,x_2,\ldots,x_n)\in\mathbb{C}^n$ such that
$$
\A x^{k-1} = \lambda x^{[k-1]},
$$
then $\lambda$ is called an \emph{eigenvalue} of $\A$, and $x$ is an \emph{eigenvector} of $\A$ corresponding to $\lambda$.
Here $\A x^{k-1} \in \mathbb{C}^n$ is defined by
$$
(\A x^{k-1})_i= \sum_{i_{2}, \ldots, i_{k} \in [n]} a_{i i_{2} \cdots i_{k}} x_{i_{2}} \cdots x_{i_{k}}, \quad i \in [n],
$$
and $x^{[k-1]}:=(x_1^{k-1},\ldots, x_n^{k-1})$.
\end{definition}

The \emph{spectral radius} of $\A$, denoted by $\rho(\A)$, is defined as the maximum modulus of the eigenvalues of $\A$.
By the Perron-Frobenius theorem of nonnegative tensors \cite{chang2008perron,friedland2013perron,YY2010,YY2011,YY2011-2},
if $\A$ is nonnegative, then $\rho(A)$ is associated with a nonnegative eigenvector.
Moreover, if $\A$ is further weakly irreducible \cite{friedland2013perron}, then $\rho(\mathcal{A})$ is the unique eigenvalue of $\mathcal{A}$ associated with a positive eigenvector.

For nonnegative and symmetric tensor, we have the following optimization result.
Let $\mathbb{R}_{+}^n$ be the set of nonnegative real vectors in $\mathbb{R}^n$.
Qi \cite{Qi2013} established the following result for nonnegative symmetric tensors, where the $\ell_k$-norm of a vector $x=(x_{1}, x_{2}, \ldots, x_{n})$ is defined by $\|x\|_k=\left(\sum_{i=1}^n |x_i|^k \right)^{1/k}$.

\begin{lem}[\cite{Qi2013}]\label{QiSpec}
Let $\A$ be an order-$k$ and dimension-$n$ nonnegative symmetric tensor, where $k\geq2$.
Then
$$
\rho(\A) = \max \bigg\{ x^\top \A x^{k-1}:=\sum_{i_{1}, \cdots, i_{k} \in [n]} a_{i_{1} i_{2} \cdots i_{k}} x_{i_{1}} x_{i_{2}} \cdots x_{i_{k}}: x\in \mathbb{R}_{+}^n, \|x\|_k=1\bigg\}.
$$
\end{lem}

Let $H$ be a $k$-graph on the vertex set $[n]$.
Its adjacency tensor $\mathcal{A}(H)=(a_{i_1\ldots i_k})$ is the order-$k$ and dimension-$n$ symmetric tensor defined by
$$
a_{i_1\ldots i_k}=
\begin{cases}
\dfrac{1}{(k-1)!},
&  \text{if~} \{i_1,\ldots,i_k\}\in E(H),\\[4pt]
~ 0, & \text{otherwise}.
\end{cases}
$$
 Clearly, $\A(H)$ is nonnegative and symmetric, and it is weakly irreducible if and only if $H$ is connected \cite{friedland2013perron, YY2011-2}.
The \emph{spectral radius} of $H$, denoted by $\rho(H)$, is the spectral radius of $\mathcal{A}(H)$.
For a subset $S\subseteq V(H)$, write $x^S:=\prod_{v\in S}x_v$.
In particular, for an edge $e\in E(H)$, we write $x^e=\prod_{v\in e}x_v$.
By Lemma~\ref{QiSpec}, we have
$$
\rho(H)=\max\left\{k\sum_{e\in E(H)}x^e: x\in\mathbb R_+^n,\ \|x\|_k=1\right\}.
$$
Consequently, there is a normalized nonnegative unit eigenvector $x$ of $H$ satisfying
\begin{equation}\label{eq:eigen-equation}
\rho(H)x_u^{k-1}= \sum_{e\in E(H): ~ u\in e} x^{e\setminus\{u\}}
\quad
\text{for every }u\in V(H).
\end{equation}
Multiplying \eqref{eq:eigen-equation} by $x_u$ and summing over $u$ gives
\begin{equation}\label{eq:rayleigh-at-eigenvector}
  \rho(H)=k\sum_{e\in E(H)}x^e.
\end{equation}

We give an upper bound and a lower bound for $\rho(S_{n,k,\beta})$.
Although the following estimate is not asymptotically tight, it is sufficient for our purpose.

\begin{lem}\label{lem:spectral-order}
Put $C_k=\frac{(k-1)^{(k-1)/k}}{(k-1)!}$.
For sufficiently large $n$, we have
\begin{equation}\label{ieq:spec-bounds}
C_k\beta^{(k-1)/k}n^{(k-1)^2/k}-O\left(n^{(k-1)^2/k-1}\right)
\leq \rho(S_{n,k,\beta})
\leq \frac{k}{(k-1)!}\beta^{(k-1)/k}n^{(k-1)^2/k}.
\end{equation}
\end{lem}

\begin{proof}
Choose $W\in\binom{[n]}{\beta}$ such that $E(S_{n,k,\beta})=\left\{e\in\binom{[n]}{k}:e\cap W\neq\emptyset\right\}$ and
define $y=(y_1,\ldots,y_n)\in\mathbb R_+^n$ by
$$
y_i=
\begin{cases}
\left(\dfrac{1}{k\beta}\right)^{1/k},
& i\in W,\\[6pt]
\left(\dfrac{k-1}{k(n-\beta)}\right)^{1/k},
& i\in[n]\setminus W.
\end{cases}
$$
Then, $\|y\|_k=1$.
By the definition of spectral radius, we have
$$
\begin{aligned}
\rho(S_{n,k,\beta})\geq k\sum_{e\in E(S_{n,k,\beta})}y^e \geq k\sum_{\substack{e\in E(S_{n,k,\beta})\\|e\cap W|=1}}y^e
&=k\beta\binom{n-\beta}{k-1}\left(\frac{1}{k\beta}\right)^{1/k}\left(\frac{k-1}{k(n-\beta)}\right)^{(k-1)/k}\\
&\geq C_k\beta^{(k-1)/k}n^{(k-1)^2/k}-O\left(n^{(k-1)^2/k-1}\right),
\end{aligned}
$$
where
\[
\binom{n-\beta}{k-1}
\sim
\frac{n^{k-1}}{(k-1)!}
\]
and hence
\[
n^{k-1} n^{-(k-1)/k}
=
n^{(k-1)^2/k}.
\]
This proves the lower bound in \eqref{ieq:spec-bounds}.

For the upper bound, let $x=(x_1,\ldots,x_n)$ be a nonnegative eigenvector corresponding to $\rho(S_{n,k,\beta})$, normalized by $\|x\|_k=1$.
Since every edge of $S_{n,k,\beta}$ intersects $W$, we have $|e\cap W|\geq1$ for every $e\in E(S_{n,k,\beta})$.
Then, we obtain
$$
\rho(S_{n,k,\beta})=k\sum_{e\in E(S_{n,k,\beta})}x^e\leq k\sum_{e\in E(S_{n,k,\beta})}|e\cap W|x^e
\leq k\sum_{w\in W}x_w\sum_{\substack{F\subseteq[n]\setminus\{w\}\\|F|=k-1}}x^F.
$$
By expanding $\left(\sum_{v\in[n]\setminus\{w\}}x_v\right)^{k-1}$ and using the nonnegativity of $x$,  we have
$$
\sum_{\substack{F\subseteq[n]\setminus\{w\}\\|F|=k-1}}x^F\leq\frac{1}{(k-1)!}\left(\sum_{v\in[n]}x_v\right)^{k-1}.
$$
By H\"older's inequality and $\|x\|_k=1$, we have
$$
\sum_{w\in W}x_w\leq\beta^{(k-1)/k}\left(\sum_{w\in W}x_w^k\right)^{1/k}\leq\beta^{(k-1)/k}
\quad\text{and}\quad
\sum_{v\in[n]}x_v\leq n^{(k-1)/k}\|x\|_k=n^{(k-1)/k}.
$$
Thus, we obtain
$$
\rho(S_{n,k,\beta})\leq\frac{k}{(k-1)!}\left(\sum_{w\in W}x_w\right)\left(\sum_{v\in[n]}x_v\right)^{k-1}
\leq\frac{k}{(k-1)!}\beta^{(k-1)/k}n^{(k-1)^2/k}.
$$
This proves the upper bound in \eqref{ieq:spec-bounds}.
\end{proof}

\section{Proof of Theorem \ref{thm:spectral-stability}}

In this section, we assume throughout that $k\geq3$, $\beta\geq2$ and $n$ is sufficiently large.
Our goal is to prove Theorem~\ref{thm:spectral-stability}.
The proof consists of two steps.

The first step is given by the following lemma.
Its proof is based on a normalized nonnegative eigenvector $x$ corresponding to $\rho(H)$.
We choose a sufficiently small constant $\eta>0$ and let $W$ be the set of vertices for which $x_v\geq\eta$.
The eigenvalue equation shows that every vertex in $W$ has large degree.
We then use this fact and the condition $\mu(H)\leq\beta$ to prove that $|W|=\beta$.
Finally, we prove that every edge of $H$ intersects $W$, which gives $H\subseteq S_{n,k,\beta}$, namely, $H$ is a sub-hypergraph of $S_{n,k,\beta}$.

\begin{lem}\label{prop:structure}
There exist constants $\delta_0>0$ and $n_0$ such that the following holds.
Let $n\geq n_0$, and let $H$ be a $k$-graph on $[n]$ satisfying
$$
\mu(H)\leq\beta\quad\text{and}\quad\rho(H)\geq(1-\delta_0)\rho(S_{n,k,\beta}).
$$
Then there exists $W\in\binom{[n]}{\beta}$ that intersects every edge of $H$. Consequently, $H\subseteq S_{n,k,\beta}$.
\end{lem}

\begin{proof}
Put $C_k=\frac{(k-1)^{(k-1)/k}}{(k-1)!}$.
Choose $\delta_0>0$ sufficiently small such that $$(1-\delta_0)\beta^{(k-1)/k}>(\beta-1)^{(k-1)/k}.$$
Let $H$ satisfy the assumptions of the proposition.
By the lower bound in \eqref{ieq:spec-bounds}, we have
\begin{equation}\label{eq:prop-rho-lower}
\rho(H)\geq(1-\delta_0)C_k\beta^{(k-1)/k}n^{(k-1)^2/k}-O\left(n^{(k-1)^2/k-1}\right).
\end{equation}

Let $x=(x_1,\ldots,x_n)$ be a nonnegative eigenvector corresponding to $\rho(H)$, normalized by $\|x\|_k=1$.
In particular, $0\leq x_v\leq1$ for every $v\in V(H)$. Choose $\eta>0$ sufficiently small such that
\begin{equation}\label{eq:eta-choice}
\frac{k^2\beta}{(k-1)!}\eta<\frac{1}{2}C_k\left((1-\delta_0)\beta^{(k-1)/k}-(\beta-1)^{(k-1)/k}\right),
\end{equation}
and set
$$
W=\{v\in V(H):x_v\geq\eta\}.
$$
By the definition of $W$ and the normalization of $x$,
$$
|W|\eta^k \leq \sum_{v\in W}x_v^k \leq \sum_{v\in V(H)}x_v^k = 1.
$$
Thus, we have $|W|\leq \eta^{-k}$.
For every $u\in W$, since $0\leq x_v\leq1$ for every $v\in V(H)$, the eigenvalue equation and \eqref{eq:prop-rho-lower} give
\begin{equation}\label{eq:large-degree}
d_H(u)\geq\sum_{\substack{e\in E(H)\\u\in e}}x^{e\setminus\{u\}}=\rho(H)x_u^{k-1}\geq(1-o(1))(1-\delta_0)C_k\beta^{(k-1)/k}\eta^{k-1}n^{(k-1)^2/k}.
\end{equation}
Thus, we have $d_H(u)=\Omega\left(n^{k-2+\frac{1}{k}}\right)$ for every $u\in W$.

We now show that $|W|\leq\beta$.
Suppose, to the contrary, that $|W|\geq\beta+1$, and choose distinct vertices $u_1,\ldots,u_{\beta+1}\in W$.
We greedily choose pairwise disjoint edges $e_1,\ldots,e_{\beta+1}$ such that $u_i\in e_i$.
Before choosing $e_i$, let $R_i=(e_1\cup\cdots\cup e_{i-1})\cup \{u_{i+1},\ldots,u_{\beta+1}\}$.
The set $R_i$ has bounded size, and the number of edges containing $u_i$ and intersecting $R_i$ is $O(n^{k-2})$.
Since
\[d_H(u_i)-O(n^{k-2})
=\Omega(n^{k-2+1/k})>0,\]
we can choose an edge $e_i$ that is disjoint from $R_i$.
This gives a matching of size $\beta+1$, contradicting $\mu(H)\leq\beta$.
Hence $|W|\leq\beta$.

We next prove $|W|=\beta$.
Suppose that $|W|\leq\beta-1$.
Let $M=\{e_1,\ldots,e_s\}$ be a maximum matching of $H$, where $s=\mu(H)\leq\beta$, and set $U=e_1\cup\cdots\cup e_s$.
Then $|U|\leq k\beta$, and the maximality of $M$ implies that every edge of $H$ intersects $U$.
Partition $E(H)$ into
$$
E_1=\{e\in E(H):e\cap W\neq\emptyset\}, \quad E_2=\{e\in E(H):e\cap W=\emptyset\}.
$$

We first estimate the contribution from $E_2$.
Since every edge in $E_2$ intersects $U\setminus W$, summing over the vertices in $U\setminus W$ contained in each edge gives
$$
k\sum_{e\in E_2}x^e\leq k\sum_{u\in U\setminus W}x_u\sum_{\substack{e\in E_2\\u\in e}}x^{e\setminus\{u\}}
\leq k\sum_{u\in U\setminus W}x_u\sum_{\substack{F\subseteq V(H)\\|F|=k-1}}x^F
\leq\frac{k}{(k-1)!}\left(\sum_{u\in U\setminus W}x_u\right)\left(\sum_{v\in V(H)}x_v\right)^{k-1},
$$
where the last inequality follows by expanding $\left(\sum_{v\in V(H)}x_v\right)^{k-1}$.
Since $x_u<\eta$ for every $u\in U\setminus W$ and $|U|\leq k\beta$, we have
$$
\sum_{u\in U\setminus W}x_u\leq k\beta\eta.
$$
On the other hand, by H\"older's inequality and $\|x\|_k=1$, we obtain
$$
\sum_{v\in V(H)}x_v\leq n^{(k-1)/k}\left(\sum_{v\in V(H)}x_v^k\right)^{1/k}=n^{(k-1)/k}.
$$
Therefore,
\begin{equation}\label{eq:E2}
k\sum_{e\in E_2}x^e\leq\frac{k^2\beta}{(k-1)!} \eta n^{(k-1)^2/k}.
\end{equation}

We next estimate the contribution from $E_1$.
Put $P=\sum_{w\in W}x_w^k$.
Then $\sum_{v\in V(H)\setminus W}x_v^k=1-P$.
The contribution from the edges in $E_1$ that intersect $W$ in exactly one vertex satisfies
$$
k\sum_{\substack{e\in E_1\\|e\cap W|=1}}x^e\leq\frac{k}{(k-1)!}\left(\sum_{w\in W}x_w\right)\left(\sum_{v\in V(H)\setminus W}x_v\right)^{k-1}.
$$
By H\"older's inequality,
$$
\sum_{w\in W}x_w \leq |W|^{(k-1)/k}P^{1/k},
\qquad
\sum_{v\in V(H)\setminus W}x_v \leq |V(H)\setminus W|^{(k-1)/k}(1-P)^{1/k}
\leq n^{(k-1)/k}(1-P)^{1/k}.
$$
Since
$$
\max_{0\leq P\leq1}P^{1/k}(1-P)^{(k-1)/k}=\frac{(k-1)^{(k-1)/k}}{k},
$$
we obtain
\begin{equation}\label{eq:E11}
k\sum_{\substack{e\in E_1\\|e\cap W|=1}}x^e\leq C_k|W|^{(k-1)/k}n^{(k-1)^2/k}.
\end{equation}

It remains to estimate the contribution from the edges in $E_1$ that intersect $W$ in at least two vertices.
Since $x_v\leq1$ for every $v\in V(H)$, the same argument gives
\begin{equation}\label{eq:E12}
k\sum_{\substack{e\in E_1\\|e\cap W|\geq2}}x^e \leq k\sum_{j=2}^{k}\binom{|W|}{j}\frac{1}{(k-j)!}n^{(k-1)(k-j)/k}=O\left(n^{(k-1)(k-2)/k}\right).
\end{equation}

Combining \eqref{eq:rayleigh-at-eigenvector}, \eqref{eq:E2}, \eqref{eq:E11}, and \eqref{eq:E12}, and using $|W|\leq\beta-1$, we obtain
$$
\rho(H)=k\sum_{e\in E_1}x^e+ k\sum_{e\in E_2}x^e
\leq \left(C_k(\beta-1)^{(k-1)/k}+\frac{k^2\beta}{(k-1)!}\eta\right)n^{(k-1)^2/k}+O\left(n^{(k-1)(k-2)/k}\right).
$$
By the choice of $\eta$, we have
$$
\begin{aligned}
C_k(\beta-1)^{(k-1)/k}+\frac{k^2\beta}{(k-1)!}\eta &<C_k(\beta-1)^{(k-1)/k}+\frac{1}{2}C_k\left((1-\delta_0)\beta^{(k-1)/k}-(\beta-1)^{(k-1)/k}\right)\\
&=\frac{1}{2}C_k\left((1-\delta_0)\beta^{(k-1)/k}+(\beta-1)^{(k-1)/k}\right)\\
&<(1-\delta_0)C_k\beta^{(k-1)/k}.
\end{aligned}
$$
The last inequality follows from
$$
(\beta-1)^{(k-1)/k} < (1-\delta_0)\beta^{(k-1)/k}.
$$
On the other hand, \eqref{eq:prop-rho-lower} gives
$$
\rho(H) \geq (1-\delta_0) C_k\beta^{(k-1)/k}n^{(k-1)^2/k}- O\left(n^{(k-1)^2/k-1}\right).
$$
Thus, the coefficient of $n^{(k-1)^2/k}$ in the upper bound for $\rho(H)$ is strictly smaller than that in the lower bound.
Since all the remaining terms have smaller order, this is a contradiction for sufficiently large $n$.
Therefore, $|W|=\beta$.

It remains to prove that every edge of $H$ intersects $W$.
Suppose, to the contrary, that there exists an edge $e'\in E(H)$ such that $e'\cap W=\emptyset$.
Write $W=\{u_1,\ldots,u_\beta\}$, and let $H'=H[V(H)\setminus e']$.
For every $u_i\in W$, at most $k\binom{n-2}{k-2}=O(n^{k-2})$ edges containing $u_i$ intersect $e'$.
Hence, by \eqref{eq:large-degree},
$$
d_{H'}(u_i) \geq d_H(u_i)-O(n^{k-2})=\Omega\left(n^{k-2+\frac{1}{k}}\right).
$$
Using the same greedy argument as above, we can choose pairwise disjoint edges $e_1,\ldots,e_\beta$ of $H'$ such that $u_i\in e_i$ for every $i\in[\beta]$. Together with $e'$, these edges form a matching of size $\beta+1$, contradicting $\mu(H)\leq\beta$.
Therefore, every edge of $H$ intersects $W$, and hence $H\subseteq S_{n,k,\beta}$.
\end{proof}

The second step is to estimate the number of edges of $ S_{n,k,\beta}$ that are missing from $H$.
The following lemma bounds this number in terms of $\rho(S_{n,k,\beta})-\rho(H)$.

\begin{lem}\label{lem:missing-edges}
There exist constants $c>0$, $\delta_1>0$, and $n_1$, where $\delta_1\leq\delta_0$, such that the
following holds. For $n\geq n_1$, let $H$ be a spanning subgraph of $S_{n,k,\beta}$ satisfying
$$
\rho(H) \geq (1-\delta_1)\rho(S_{n,k,\beta}).
$$
Then
$$
\rho(S_{n,k,\beta})-\rho(H) \geq c\left|E(S_{n,k,\beta})\setminus E(H)\right| n^{-(k-1)/k}.
$$
\end{lem}

\begin{proof}
Put $W=[\beta]$. 
Then, by definition, every edge of $H$ intersects $W$.
%By Lemma~\ref{prop:structure}, there exists $W\in\binom{[n]}{\beta}$ such that
%After relabeling the vertices, we may assume that.
%Then $H$ is a spanning subgraph of $S_{n,k,\beta}$.
Let $x=(x_1,\ldots,x_n)$ be a nonnegative eigenvector corresponding to $\rho(H)$, normalized by $\|x\|_k=1$, and put $P=\sum_{w\in W}x_w^k$.
Then $0\leq P\leq1$ and $1-P=\sum_{v\in[n]\setminus W}x_v^k$.

For $1\leq j\leq\min\{k,\beta\}$, Maclaurin's inequality and the power mean inequality give
$$
\sum_{A \in \binom{W}{j}}x^A \leq \binom{\beta}{j}\left(\frac{1}{\beta}\sum_{w\in W}x_w\right)^j \leq \binom{\beta}{j}\left(\frac{P}{\beta}\right)^{j/k}.
$$
Similarly,
$$
\sum_{B \in \binom{[n]\setminus W}{k-j}}x^B \leq \binom{n-\beta}{k-j}\left(\frac{1}{n-\beta}\sum_{v\in[n]\setminus W}x_v\right)^{k-j} \leq \binom{n-\beta}{k-j}\left(\frac{1-P}{n-\beta}\right)^{(k-j)/k}.
$$
Since $H\subseteq S_{n,k,\beta}$, we obtain
$$
\rho(H)=k\sum_{e\in E(H)}x^e \leq k\sum_{e\in E(S_{n,k,\beta})}x^e \leq k\sum_{j=1}^{\min\{k,\beta\}}\binom{\beta}{j}\binom{n-\beta}{k-j}\left(\frac{P}{\beta}\right)^{j/k}\left(\frac{1-P}{n-\beta}\right)^{(k-j)/k}.
$$
Denote the expression on the right, regarded as a function of $P\in[0,1]$, by $F_n(P)$.
Then, the above inequality can be written as $\rho(H)\leq F_n(P)$.
Define $y=(y_1,\ldots,y_n)\in\mathbb R_+^n$ by
$$
y_v=
\begin{cases}
\left(\dfrac{P}{\beta}\right)^{1/k},
& v\in W,\\[6pt]
\left(\dfrac{1-P}{n-\beta}\right)^{1/k},
& v\in[n]\setminus W.
\end{cases}
$$
Then $\|y\|_k=1$ and $k\sum_{e\in E(S_{n,k,\beta})}y^e=F_n(P)$.
Therefore, by the definition of the tensor spectral radius, we have
\begin{equation}\label{eq:Fn-upper}
F_n(P)\leq\rho(S_{n,k,\beta})
\end{equation}
for every $P\in[0,1]$.

The term with $j=1$ gives the main term of $F_n(P)$, while the sum of the terms with $j\geq2$ is $O(n^{(k-1)(k-2)/k})$.
It follows that
\begin{equation}\label{eq:Fn-limit}
n^{-(k-1)^2/k}F_n(P)=\frac{k\beta^{(k-1)/k}}{(k-1)!}P^{1/k}(1-P)^{(k-1)/k}+o(1),
\end{equation}
where the term $o(1)$ does not depend on $P\in[0,1]$.

We claim that
\begin{equation}\label{eq:P-bounded}
\frac{1}{2k}<P<1-\frac{1}{2k}.
\end{equation}
Suppose, to the contrary, that $P\in[0,1/(2k)]\cup[1-1/(2k),1]$.
Note that the function $P^{1/k}(1-P)^{(k-1)/k}$ has a unique maximum at $P=1/k$.
%Hence, by \eqref{eq:Fn-limit}, if $P\in[0,1/(2k)]\cup[1-1/(2k),1]$,
We may choose $\delta_1>0$ sufficiently small and $n_1$ sufficiently large such that, for every $n\geq n_1$,
$$
F_n(P)\leq(1-2\delta_1)F_n(1/k).
$$
By \eqref{eq:Fn-upper} and the assumption on $\rho(H)$, we have
$$
\rho(H)\leq F_n(P)\leq(1-2\delta_1)F_n(1/k)\leq(1-2\delta_1)\rho(S_{n,k,\beta})<(1-\delta_1)\rho(S_{n,k,\beta})\leq\rho(H),
$$
a contradiction.

For $1\leq j\leq\min\{k,\beta\}$, let
$$
E_j=\{e\in E(H):|e\cap W|=j\},
$$
and let $m_j$ be the number of edges $e\in E(S_{n,k,\beta})\setminus E(H)$ satisfying $|e\cap W|=j$.
Since $H\subseteq S_{n,k,\beta}$, we have
$$
m_j=\binom{\beta}{j}\binom{n-\beta}{k-j}-|E_j|.
$$
By the definition of $E_j$ and Maclaurin's inequality,
$$
\sum_{e\in E_j}(x^e)^k \leq \left(\sum_{A \in \binom{W}{j}}(x^A)^k\right)
\left(\sum_{B \in \binom{[n]\setminus W}{k-j}}(x^B)^k\right) \leq \binom{\beta}{j}\binom{n-\beta}{k-j}\left(\frac{P}{\beta}\right)^j\left(\frac{1-P}{n-\beta}\right)^{k-j}.
$$
By H\"older's inequality,
$$
\sum_{e\in E_j}x^e \leq |E_j|^{(k-1)/k}\left(\sum_{e\in E_j}(x^e)^k\right)^{1/k} \leq |E_j|^{(k-1)/k}\left[\binom{\beta}{j}\binom{n-\beta}{k-j}\left(\frac{P}{\beta}\right)^j\left(\frac{1-P}{n-\beta}\right)^{k-j}\right]^{1/k}.
$$
Since $|E_j|=\binom{\beta}{j}\binom{n-\beta}{k-j}-m_j$, applying $(1-s)^{(k-1)/k}\leq1-\frac{k-1}{k}s$ for $0\leq s\leq1$ gives
$$
\binom{\beta}{j}\binom{n-\beta}{k-j}\left(\frac{|E_j|}{\binom{\beta}{j}\binom{n-\beta}{k-j}}\right)^{(k-1)/k} \leq \binom{\beta}{j}\binom{n-\beta}{k-j}-\frac{k-1}{k}m_j.
$$
Since $E(H)=\bigcup_{j=1}E_j$, we obtain
\begin{equation}\label{eq:missing-loss}
\begin{aligned}
\rho(H) &= k\sum_{j=1}^{\min\{k,\beta\}}\sum_{e\in E_j}x^e\\
&\leq k\sum_{j=1}^{\min\{k,\beta\}}|E_j|^{(k-1)/k}\left[\binom{\beta}{j}\binom{n-\beta}{k-j}\left(\frac{P}{\beta}\right)^j\left(\frac{1-P}{n-\beta}\right)^{k-j}\right]
^{1/k}\\
&=
k\sum_{j=1}^{\min\{k,\beta\}}\binom{\beta}{j}\binom{n-\beta}{k-j}\left(\frac{|E_j|}{\binom{\beta}{j}\binom{n-\beta}{k-j}}\right)^{(k-1)/k}
\left(\frac{P}{\beta}\right)^{j/k}\left(\frac{1-P}{n-\beta}\right)^{(k-j)/k}\\
&\leq
k\sum_{j=1}^{\min\{k,\beta\}}\left[\binom{\beta}{j}\binom{n-\beta}{k-j}-\frac{k-1}{k}m_j\right]\left(\frac{P}{\beta}\right)^{j/k}\left(\frac{1-P}{n-\beta}\right)
^{(k-j)/k}\\
&=
F_n(P)-(k-1)\sum_{j=1}^{\min\{k,\beta\}}m_j\left(\frac{P}{\beta}\right)^{j/k}\left(\frac{1-P}{n-\beta}\right)^{(k-j)/k}.
\end{aligned}
\end{equation}

By \eqref{eq:P-bounded} and $n-\beta\leq n$, for every $1\leq j\leq\min\{k,\beta\}$,
$$
\left(\frac{P}{\beta}\right)^{j/k}\left(\frac{1-P}{n-\beta}\right)^{(k-j)/k} \geq \left(\frac{1}{2k\beta}\right)^{j/k}\left(\frac{1}{2kn}\right)^{(k-j)/k} \geq \frac{1}{2k\beta}n^{-(k-1)/k}.
$$
Combining \eqref{eq:Fn-upper} and \eqref{eq:missing-loss}, we obtain
$$
\rho(H)\leq\rho(S_{n,k,\beta})-\frac{k-1}{2k\beta}n^{-(k-1)/k}\sum_{j=1}^{\min\{k,\beta\}}m_j.
$$
Finally,
$$
\sum_{j=1}^{\min\{k,\beta\}}m_j
=
\left|E(S_{n,k,\beta})\setminus E(H)\right|.
$$
Therefore,
$$
\rho(S_{n,k,\beta})-\rho(H)\geq\frac{k-1}{2k\beta}\left|E(S_{n,k,\beta})\setminus E(H)\right|n^{-(k-1)/k}.
$$
The proof is complete by taking $c=(k-1)/(2k\beta)$ and choosing $n_1$ sufficiently large.
\end{proof}

We now combine Lemma~\ref{prop:structure} and Lemma~\ref{lem:missing-edges} to prove Theorem~\ref{thm:spectral-stability}.

\begin{proof}[\bf Proof of Theorem~\ref{thm:spectral-stability}]
Let $\delta\geq0$ be sufficiently small such that $\delta\leq\min\{\delta_0,\delta_1\}$.
Since
$$
\rho(H)\geq (1-\delta)\rho(S_{n,k,\beta})\geq (1-\delta_0)\rho(S_{n,k,\beta}),
$$
Proposition~\ref{prop:structure} gives a set $W\in\binom{[n]}{\beta}$ such that $H\subseteq S_{n,k,\beta}$.
Since $\delta\leq\delta_1$, Lemma~\ref{lem:missing-edges} gives
$$
c\left|E(S_{n,k,\beta})\setminus E(H)\right|n^{-(k-1)/k}\leq \rho(S_{n,k,\beta})-\rho(H)\leq \delta\rho(S_{n,k,\beta}).
$$
By Lemma~\ref{lem:spectral-order}, there exists a constant $C_0=C_0(k,\beta)>0$ such that
$$
\rho(S_{n,k,\beta})\leq C_0n^{(k-1)^2/k}.
$$
Consequently,
$$
\left|E(S_{n,k,\beta})\setminus E(H)\right| \leq \frac{C_0}{c}\delta n^{k-1}.
$$
Finally, since $H\subseteq S_{n,k,\beta}$,
$$
\left|E(H)\triangle E(S_{n,k,\beta})\right| = \left|E(S_{n,k,\beta})\setminus E(H)\right|.
$$
The proof follows by taking $C=C_0/c$.
\end{proof}

\begin{proof}[\bf Proof of Corollary~\ref{cor:spec-emc}]
Suppose, to the contrary, that the desired inequality fails; that is $\rho(H)>\rho(S_{n,k,\beta})$.
Applying Theorem~\ref{thm:spectral-stability} with $\delta=0$, we obtain a set $W\in\binom{[n]}{\beta}$ such that every edge of $H$ intersects $W$. After relabeling the vertices of $H$ so that $W=[\beta]$, we have $\left|E(H)\triangle E(S_{n,k,\beta})\right|=0$.
It follows that $H\cong S_{n,k,\beta}$, and hence $\rho(H)=\rho(S_{n,k,\beta})$, a contradiction.
Therefore,
$$
\rho(H)\leq\rho(S_{n,k,\beta}).
$$
If equality holds, Theorem~\ref{thm:spectral-stability} with $\delta=0$ also gives $H\cong S_{n,k,\beta}$.
Conversely, if $H\cong S_{n,k,\beta}$, then equality clearly holds.
\end{proof}

\end{document}